\documentclass[10pt,a4paper]{amsart}
\usepackage{amssymb}
\usepackage{amsmath}
\usepackage{amsfonts}
\usepackage[english]{babel}
\usepackage[T1]{fontenc}
\usepackage[latin1]{inputenc}
\usepackage{amsthm}
\usepackage[all]{xy}
\usepackage{a4wide}

\usepackage{array}
\usepackage{rotating}
\usepackage{booktabs}

\numberwithin{equation}{section}

\theoremstyle{definition}
\newtheorem{thm}{Theorem}[section]
\newtheorem{lemma}[thm]{Lemma}
\newtheorem{cor}[thm]{Corollary}

\newtheorem{rem}[thm]{Remark}

\begin{document}

\begin{abstract}
We compute the spectra and the essential spectra of bounded linear fractional composition operators acting on the Hardy and weighted Bergman spaces of the upper half-plane. We are also able to extend the results to weighted Dirichlet spaces of the upper half-plane.
\end{abstract}

\thanks{The author is supported by the Magnus Ehrnrooth Foundation in Finland.}
\subjclass[2010]{47B33}
\keywords{Composition operator, spectrum, essential spectrum, Hardy space of the upper half-plane, weighted Bergman spaces of the upper half-plane}

\title[Spectra of composition operators on spaces of the upper half-plane]{Spectra of linear fractional composition operators on the Hardy and weighted Bergman spaces of the half-plane}

\author[R. Schroderus]{Riikka Schroderus}
\address{Department of Mathematics and Statistics\\ University of Helsinki, Box 68 \\ FI-00014 Helsinki, Finland}
\email{riikka.schroderus@helsinki.fi}

\maketitle

\section{Introduction}\label{intro}

The Hardy space $H^2(\Pi^+)$ and the weighted Bergman spaces $\mathcal{A}_{\alpha}^2 (\Pi^+ )$ of the upper half-plane $\Pi^+ $ are reasonably well understood Hilbert spaces of analytic functions (the spaces are defined in Section \ref{prelim}). \textit{Composition operators} $ C_{\tau} : f\mapsto f\circ \tau $, where $ \tau : \Pi^+ \longrightarrow \Pi^+$ is analytic, acting on these spaces are, however, much less studied when compared to their counterparts in the unit disc setting. Matache \cite{Ma1} found a condition for boundedness of $C_{\tau}$ on $H^2(\Pi^+) $ in terms of Carleson measures and showed later in \cite{Ma2} that there are no compact composition operators on $H^2(\Pi^+)$. In \cite{SS} Shapiro and Smith extended the non-compactness result to $\mathcal{A}_{\alpha}^2 (\Pi^+ )$. Boundedness of a composition operator $C_{\tau}$ on the Hardy or the weighted Bergman spaces of the half-plane has been proved to be equivalent with the angular derivative of the inducing map at infinity, denoted by $\tau' (\infty)$, being finite and positive (see \cite{Ma3, EW}). In addition, Elliott et al. \cite{EW, EJ} have shown that whenever $C_{\tau}$ is bounded, the operator norm, the essential operator norm and the spectral radius are all equal and determined by the quantity $\tau'(\infty )$. The above properties show that the spaces $H^2(\Pi^+)$ and $\mathcal{A}_{\alpha}^2 (\Pi^+ )$ differ significantly from their unit disc analogues $H^2 (\mathbb{D})$ and $\mathcal{A}_{\alpha}^2 (\mathbb{D})$ regarding the composition operators acting on them. We refer e.g. to \cite{Sh, CM} for expositions of the rich theory of composition operators on spaces defined on the unit disc.

It is natural to ask what the spectral properties of composition operators acting on the half-plane are. In fact, in the unit disc setting the spectral picture of composition operators has been completely determined when the inducing maps are \textit{linear fractional transformations} and the operators act on weighted Dirichlet spaces, including $H^2 (\mathbb{D})$ and $\mathcal{A}_{\alpha}^2 (\mathbb{D})$ for all $\alpha >-1$ (see, for instance, \cite{No, CM, Hu, Hi, Po, GS}). The spectra in the corresponding setting on the half-plane are largely unknown; in the (unweighted) Dirichlet space of $\Pi^+$ the spectra are known but besides that only the spectra (and the essential spectra) for invertible or self-adjoint parabolic and invertible hyperbolic composition operators acting on the Hardy space $H^2(\Pi^+)$ have been computed (see \cite{Ma4}). In contrast to the unit disc case, not all linear fractional transformations $\tau$ induce bounded composition operators on $H^2(\Pi^+)$ or $\mathcal{A}_{\alpha}^2 (\Pi^+ )$. Indeed, $C_{\tau}$ is bounded only when $\tau$ is a parabolic or a hyperbolic self-map of $ \Pi^+$ fixing infinity. In this paper we compute the spectra and the essential spectra of these composition operators.

In the parabolic case (see Theorem \ref{sppara} in Section \ref{secpara}) we obtain the following result: 

\vspace{5pt}

\textbf{Theorem A.} Let $\tau$ be a parabolic self-map of $\Pi^+$, that is, $\tau (w)=w+w_0$, where $ \textnormal{Im}\, w_0 \geq 0$ and $w_0 \neq 0$. Then the spectrum of $C_{\tau}$ acting on the Hardy or the weighted Bergman spaces of the upper half-plane equals 
$$ \overline{ \{ e^{iw_0t} : t\in [0,\infty ) \} } = \left\{ \begin{array}{ll}
\mathbb{T}, & \textnormal{ when } w_0\in \mathbb{R},\\
\{ e^{iw_0t} : t\in [0,\infty ) \} \cup \{ 0\}, & \textnormal{ when } w_0\in \Pi^+.
\end{array}
\right.
$$
Moreover, the essential spectrum coincides with the spectrum in both cases.

\vspace{5pt}

For a moment, write formally $H^2(\Pi^+) = \mathcal{A}_{-1}^2 (\Pi^+ ) $. We can summarize the results in the hyperbolic case (see Theorems \ref{sphypauto}, \ref{hyptauyks}, \ref{hyptaukaks} and Corollary \ref{sphypautoinverse} in Section \ref{hypsec}) as follows:

\vspace{5pt}

\textbf{Theorem B.} Let $\tau$ be a hyperbolic self-map of $\Pi^+$, that is, $\tau (w)=\mu w+w_0$, where $\mu\in (0,1)\cup (1,\infty)$ and $\textnormal{Im}\, w_0 \geq 0$. Then, for all $\alpha \geq -1$, the spectrum of $C_{\tau}$ acting on $\mathcal{A}_{\alpha}^2 (\Pi^+ )$ is

\vspace{5pt}

\begin{itemize}
\item[i)] $  \big\{ \lambda \in \mathbb{C} : \vert \lambda \vert = \mu^{ -(\alpha  +2 )/2}\big\} $, when $w_0 \in \mathbb{R}$,

\vspace{5pt}

\item[ii)] $ \big\{ \lambda \in \mathbb{C} : \vert \lambda \vert\leq \mu^{ -(\alpha  +2 )/2}\big\}$, when $w_0\in \Pi^+$.
\end{itemize}
Moreover, the essential spectrum coincides with the spectrum in both cases.

\vspace{5pt}

From these results we are also able to deduce the spectra of the parabolic and the hyperbolic composition operators on weighted Dirichlet spaces $\mathcal{D}_{\alpha}^2 (\Pi^+ ) $ for $\alpha >-1$ (see Theorems \ref{dirpara} and \ref{dirhyper} in Section \ref{secdir}).

There are similarities as well as differences when we consider the composition operators $C_{\tau}$ on the half-plane and $C_{\varphi}$ on the unit disc, where the inducing maps satisfy $\tau =h \circ \varphi \circ h^{-1}$ for some conformal map $h: \mathbb{D} \longrightarrow \Pi^+$. It is worth noting that the composition map $ C_h $ which gives a unitary equivalence in the (unweighted) Dirichlet space setting between $ C_{\tau}: \mathcal{D}_0^2 (\Pi^+)\longrightarrow \mathcal{D}_0^2 (\Pi^+)$ and $C_{\varphi}: \mathcal{D}^2 (\mathbb{D})/\mathbb{C} \longrightarrow \mathcal{D}^2 (\mathbb{D})/\mathbb{C} $ does not give even similarity in the Hardy or the weighted Bergman spaces since $C_h$ is not bounded below. Therefore, in general, we cannot deduce the spectra of $C_{\tau}$ on the half-plane from the spectra of $C_{\varphi}$ on the unit disc. Nevertheless, in the parabolic case the spectra are the same. The hyperbolic composition operators, on the other hand, seem to have a life of their own when comparing the spectral results between $\mathbb{D}$ and $\Pi^+$.

\section{Preliminaries and notation}\label{prelim}

Throughout the paper we will use the following notations: $\overline{\mathbb{C}}:=\mathbb{C}\cup \{\infty\}$ denotes the extended complex plane, $\Pi^+ := \{ w\in \mathbb{C}: \textnormal{Im}\, w >0\}$ the upper half-plane, $\mathbb{D}:=\{z\in \mathbb{C} : \vert z\vert <1\}$ the unit disc and $\mathbb{T}:= \{z\in \mathbb{C} : \vert z\vert =1\} $ the unit circle.

\subsection*{The Hardy and weighted Bergman spaces of the upper half-plane}
The \textit{Hardy space} $H^2 (\Pi^{+})$ of the upper half-plane consists of the analytic functions $F:\Pi^{+} \longrightarrow \mathbb{C} $ such that
$$ \Vert F\Vert_{H^2 (\Pi^{+})} = \sup_{y>0} \Big(\int_{-\infty}^{\infty} \vert F(x+iy)\vert^2 \, dx \Big)^{1/2} < \infty .$$
The space $H^2 (\Pi^{+})$ is isometrically embedded into $ L^2 (\mathbb{R})$ via the mapping $F \longmapsto F^{*}$, where $F^{*}(x)= \lim_{y \longrightarrow 0^{+}} F(x+iy)$ (for more information on Hardy spaces of the half-plane, see \cite[Chapter 11]{Du}, for instance).

For $\alpha > -1$, the \textit{weighted Bergman space} $\mathcal{A}_{\alpha}^2 (\Pi^{+})$ of the upper half-plane consists of the functions $F$ analytic on $\Pi^{+}$ satisfying
$$ \Vert F\Vert_{\mathcal{A}_{\alpha}^2 (\Pi^{+})}  = \Big( \int_{\Pi^{+}} \vert F(x+iy)\vert^2 \, y^{\alpha} \, dx \, dy \Big)^{1/2} < \infty .$$

The Hardy space $H^2 (\Pi^{+})$ can often be formally interpreted as the ``limit case'' of the weighted Bergman spaces as $\alpha \longrightarrow -1$, that is $ H^2 (\Pi^{+}) = \mathcal{A}_{-1}^2 (\Pi^{+})$. In the sequel, we will use this formal convention. The spaces $\mathcal{A}_{\alpha}^2 (\Pi^{+})$, for $\alpha \geq -1$, are Hilbert spaces and we use the notation $\langle \cdot , \cdot \rangle_{\mathcal{A}_{\alpha}^2 (\Pi^{+}) }$ for the inner product.

\vspace{5pt}

Often we will find it useful to change the perspective from $\mathcal{A}_{\alpha}^2 ( \Pi^+)$ to the corresponding space in the unit disc, namely to the classical Hardy space
$$H^2 (\mathbb{D})=  \big\{f: \mathbb{D} \longrightarrow \mathbb{C} \textnormal{ analytic}, \Vert f \Vert^2= \sup_{0 <r <1 }\int_{0}^{2\pi} \vert f(re^{i\theta})\vert^2 \, \frac{d\theta}{2\pi} <\infty \big\},$$ 
or the weighted Bergman spaces defined for all $\alpha > -1$ by (for $z=u+iv$)
$$\mathcal{A}_{\alpha}^2 (\mathbb{D})= \{f: \mathbb{D} \longrightarrow \mathbb{C} \textnormal{ analytic}, \Vert f \Vert_{\alpha}^2= \int_{\mathbb{D}} \vert f(z)\vert^2 (1-\vert z\vert^2)^{\alpha}\, du \, dv <\infty \big\} .$$
We will write also here $H^2 (\mathbb{D}) =  \mathcal{A}_{-1}^2 (\mathbb{D})$ for convenience. Recall that the function $h$, where 
\begin{equation}\label{hoo}
h(z)=i \frac{1+z}{1-z}, 
\end{equation} 
is a conformal map from $\mathbb{D}$ onto $\Pi^+$ with the inverse $h^{-1}(w)=\frac{w-i}{w+i}$. For any $\alpha \geq -1$ there is an isometric isomorphism (see e.g. \cite[pp. 128-131]{Ho} for the Hardy space and \cite{DGM} for the weighted Bergman spaces)
\begin{equation}\label{jii}
J: \mathcal{A}_{\alpha}^2 (\mathbb{D}) \longrightarrow \mathcal{A}_{\alpha}^2 ( \Pi^+),  \quad (Jf)(w)=  f \Big( \frac{w-i}{w+i}\Big)\frac{c_{\alpha}}{(w+i)^{\alpha +2}},
\end{equation} 
where
$$c_{\alpha}=\left\{ \begin{array}{ll}
1/\sqrt{\pi}, & \textnormal{ when } \alpha =-1, \\
2^{\alpha +1}, & \textnormal{ when } \alpha > -1 .
\end{array} \right.
$$
The inverse of $J$ is given by
$$ J^{-1}: \mathcal{A}_{\alpha}^2 ( \Pi^+) \longrightarrow \mathcal{A}_{\alpha}^2 (\mathbb{D}) ,  \quad (J^{-1}F)(z)= F \Big( i\frac{1+z}{1-z}\Big)\frac{d_{\alpha}}{(1-z)^{\alpha +2}},
$$
where
$$d_{\alpha}=\left\{ \begin{array}{ll}
2\sqrt{\pi}i, & \textnormal{ when } \alpha =-1, \\
2i, & \textnormal{ when } \alpha > -1 .
\end{array} \right.
$$

\vspace{5pt}
 
Recall that the \textit{reproducing kernels} of $\mathcal{A}_{\alpha}^2 ( \Pi^+)$, for $\alpha \geq -1$, are functions $K_{w_0}^{\alpha} \in \mathcal{A}_{\alpha}^2 ( \Pi^+)$ satisfying 
$$ \langle F, K_{w_0}^{\alpha}\rangle_{\mathcal{A}_{\alpha}^2 ( \Pi^+)} = F(w_0) \textnormal{ for any } F \in \mathcal{A}_{\alpha}^2 ( \Pi^+) \textnormal{ and } w_0 \in \Pi^+ .$$ 
We will next compute the explicit form of $ K_{w_0}^{\alpha}$ since we could not find a suitable reference for the general case.

\begin{lemma} 
The reproducing kernels $ K_{w_0}^{\alpha} \in  \mathcal{A}_{\alpha}^2 ( \Pi^+) $ for $\alpha \geq -1$ are of the form
\begin{equation}\label{reprkernela} K_{w_0}^{\alpha} (w)= \frac{k_{\alpha}}{ (w-\overline{w_0})^{\alpha +2}},
\end{equation}
where
$$k_{\alpha}=\left\{ \begin{array}{ll}
i/2\pi, & \textnormal{ when } \alpha =-1, \\
i(\alpha +1) 2^{\alpha }, & \textnormal{ when } \alpha > -1.
\end{array} \right.
$$
\end{lemma}

\begin{proof}
Recall first that for $\alpha \geq -1$, the reproducing kernels $g_{z_0}^{\alpha}$ of $\mathcal{A}_{\alpha}^2 (\mathbb{D})$ at a point $z_0\in \mathbb{D}$ are of the form (see \cite[Chapter 2]{CM}, for instance)
$$g_{z_0}^{\alpha} (z)= \frac{ \nu_{\alpha}}{(1-\overline{z_0}z)^{\alpha +2}}, \textnormal{ where } \nu_{-1}=1  \textnormal{ and }\nu_{\alpha }= (\alpha +1)   \textnormal{ otherwise} .$$
But for each $F\in \mathcal{A}_{\alpha}^2 ( \Pi^+)$ and $ w_0 = i \frac{1+z_0}{1-z_0}\in \Pi^+$ it holds that
$$F(w_0)= \frac{(1-z_0)^{\alpha +2}}{d_{\alpha}} (J^{-1}F)(z_0) =  \frac{(1-z_0)^{\alpha +2}}{d_{\alpha}} \langle J^{-1}F, g_{z_0}^{\alpha}\rangle_{\mathcal{A}_{\alpha}^2 (\mathbb{D})} =\frac{(1-z_0)^{\alpha +2}}{d_{\alpha}}  \langle F, Jg_{z_0}^{\alpha}\rangle_{\mathcal{A}_{\alpha}^2 ( \Pi^+)} ,$$ 
where $J$ is the isometric isomorphism in \eqref{jii}. On the other hand, $F(w_0)=\langle F, K_{w_0}^{\alpha} \rangle_{ \mathcal{A}_{\alpha}^2 ( \Pi^+)} $, so that
\begin{equation*}
\begin{split}
K_{w_0}^{\alpha} (w)& =\frac{(1-\overline{z_0})^{\alpha +2}}{\overline{d_{\alpha}}}  \big(Jg_{z_0}^{\alpha}\big)(w) = \frac{(1-\overline{z_0})^{\alpha +2}}{\overline{d_{\alpha}}}\frac{\nu_{\alpha}}{\Big(1-\overline{z_0}\big( \frac{w-i}{w+i}\big)\Big)^{\alpha +2}} \frac{c_{\alpha}}{(w+i)^{\alpha +2}} \\
&= \frac{\nu_{\alpha} c_{\alpha}}{\overline{d_{\alpha}}} \frac{1}{ \Big(w- \overline{ i \frac{1+z_0}{1-z_0}}\Big)^{\alpha +2}} = \frac{k_{\alpha}}{ (w-\overline{w_0})^{\alpha +2}}.
\end{split}
\end{equation*} 

\end{proof}
\vspace{5pt}

For any $\beta \geq 0$, let us denote by $ L_{\beta}^2$ the space consisting of the measurable functions $f:\mathbb{R}_+ \longrightarrow \mathbb{C} $ with finite norm
$$  \Vert f\Vert_{L_{\beta}^2} = \left\{ \begin{array}{ll}
   \Big( 2\pi\int_{0}^{\infty} \vert f(t)\vert^2 \, dt \Big)^{1/2} , & \textnormal{ when } \beta = 0,\\
  \Big(\frac{2\pi \Gamma (\beta)}{2^{\beta}}\int_{0}^{\infty} \vert f(t)\vert^2 t^{-\beta} \, dt \Big)^{1/2} , & \textnormal{ when } \beta > 0.
  \end{array} \right.
$$

For future reference we will write explicitly the classical Paley-Wiener theorem (\cite{PW}, see also \cite[Thm. 11.9]{Du} or \cite[Thm. 19.2]{R}) and its generalization to the weighted Bergman spaces (a detailed proof for all $\alpha >-1$ can be found in \cite[Thm. 1]{DGM}) by using the above notation: For any $\alpha \geq -1$, the space $\mathcal{A}_{\alpha}^2 ( \Pi^+)$ is isometrically isomorphic to $L_{\alpha +1}^2$ under the Fourier transform $\mathcal{F}$. Indeed, $F \in \mathcal{A}_{\alpha}^2 ( \Pi^+)$ if and only if there exists a function $f\in L_{\alpha +1}^2$ such that
$$ F(w)= (\mathcal{F}^{-1} f)(w) = \int_0^{\infty} f(t) e^{iwt} \, dt , \quad w\in \Pi^+ .$$
Moreover,
\begin{equation}\label{pwnormeq}
\Vert F\Vert_{\mathcal{A}_{\alpha}^2 ( \Pi^+)}^2 =  \Vert f\Vert_{L_{\alpha +1}^2}^2 = b_{\alpha} \int_0^{\infty} \vert f(t) \vert^2 t^{-(\alpha +1)}  \, dt,
\end{equation}
where 
$$b_{\alpha}=\left\{ \begin{array}{ll}
2\pi, & \textnormal{ when } \alpha =-1, \\
2^{-\alpha}\pi \Gamma (\alpha +1), & \textnormal{ when } \alpha > -1.
\end{array} \right.
$$

\subsection*{Composition operators $C_{\tau}$ on $\mathcal{A}_{\alpha}^2 ( \Pi^+)$ for $\alpha \geq -1$}
Let $\tau : \Pi^+\longrightarrow \Pi^+$ be an analytic map. Denote by $C_{\tau}$ the composition map induced by $\tau$, 
$$C_{\tau} F = F \circ \tau , \quad F:\Pi^+ \longrightarrow\mathbb{C} \textnormal{ analytic} .$$
(The definition can be generalized by replacing $\Pi^+$ by any open subset of $\mathbb{C}$.)

\vspace{5pt}

In what follows we will consider the composition operator $C_{\tau}$ acting on $\mathcal{A}_{\alpha}^2 ( \Pi^+)$ for $\alpha \geq -1$, where $\tau$ is a linear fractional self-map of $\Pi^+ $ such that $\tau (w)\not\equiv w$. Recall that, in general, the linear fractional transformations (abbreviated by LFT) $\psi$ are meromorphic bijections $ \overline{\mathbb{C}} \longrightarrow \overline{\mathbb{C}} $ and they can be written in the form
$$\psi (z)= \frac{az+b}{cz+d}, \quad \textnormal{where } a,b,c,d \in \mathbb{C}, \, ad-bc\neq 0 .$$
The non-identity LFTs have exactly two \textit{fixed points}, i.e. points $z \in \overline{\mathbb{C}}$ such that $\psi (z)=z$, counting multiplicities. According to their behaviour at the fixed points, the LFTs can be classified into \textit{parabolic, hyperbolic, elliptic} or \textit{loxodromic} maps. For more information on LFTs, see \cite[Chapter 0]{Sh}, for instance.

\vspace{5pt}

A necessary and sufficient condition for the boundedness of $C_{\tau}$ (for any analytic $\tau : \Pi^+ \longrightarrow \Pi^+$) on $\mathcal{A}_{\alpha}^2 ( \Pi^+) $, for $\alpha \geq -1$, is that $\tau$ has a finite positive \textit{angular derivative} at infinity (see \cite[Thm. 15]{Ma3} or \cite[Thm. 3.1]{EJ} for the Hardy space and \cite[Thm. 3.4]{EW} for the weighted Bergman spaces). That is, $\tau$ must fix $\infty$ and the non-tangential limit is
$$\lim_{w \longrightarrow\infty} \frac{w}{\tau (w)} =: \tau' (\infty) \in (0,\infty).$$

It is a standard exercise to show that the linear fractional self-maps of $\Pi^+$ that fix $\infty$ are precisely of the form
\begin{equation}\label{tau} \tau (w)= \mu w + w_0, \textnormal{ where } \mu >0 \textnormal{ and Im}\, w_0\geq 0. 
\end{equation}
Depending on the value of $\mu$ in Equation \eqref{tau}, we have two different types of mappings that induce bounded composition operators on $\mathcal{A}_{\alpha}^2 ( \Pi^+)$ for $\alpha \geq -1$:
\begin{itemize}
\item[1.] $\tau$ is \textit{parabolic} when $\mu =1$ (and necessarily $w_0 \neq 0$). For a parabolic mapping it always holds that the derivative at its only fixed point is $1$, that is $\tau' (\infty)=1$.
\item[2.] $\tau$ is \textit{hyperbolic} when $\mu \in (0,1)\cup (1,\infty)$ and $\textnormal{Im}\, w_0 \geq 0 $. In this case, by definition $ \tau' (\infty)=1/\mu$. It holds that if $p$ and $q$ are, respectively, the attractive and repulsive fixed points of a hyperbolic mapping $\tau$, then
\begin{equation}\label{attrepder}
\tau' (p)= \frac{1}{\tau' (q)}.
\end{equation}
\end{itemize}
Note that in both cases $\tau$ is an automorphism of $ \Pi^+$ if and only if $\textnormal{Im}\, w_0 =0$. 

\vspace{5pt}

The condition that $\tau$ must fix $\infty$ for $C_{\tau}$ to be bounded on $\mathcal{A}_{\alpha}^2 ( \Pi^+)$ for $\alpha \geq -1$ excludes the elliptic, loxodromic and certain hyperbolic mappings from inducing composition operators on the half-plane. The counterparts of these mappings in the unit disc setting induce diagonal (the elliptic case) or compact (the loxodromic and hyperbolic non-automorphisms having fixed points in $\mathbb{D}$ and $\overline{\mathbb{C}}\setminus \overline{\mathbb{D}}$) composition operators, for instance, on $ \mathcal{A}_{\alpha}^2 ( \mathbb{D})$ for $\alpha \geq -1$.

\subsection*{Spectra of linear operators}
We recall some basic definitions and facts from general spectral theory that are used in the sequel. A suitable reference for the spectral theory of linear operators on Hilbert spaces is \cite{Mu}, for instance. Let $T$ be a bounded linear operator acting on a Hilbert space $\mathcal{H}$. The \textit{spectrum} of $T: \mathcal{H} \longrightarrow  \mathcal{H}$ is defined by
$$\sigma \big(T; \mathcal{H}\big)= \{ \lambda \in \mathbb{C} : T-\lambda \textnormal{ is not invertible on } \mathcal{H}\}.$$
We will also use the notation $\sigma (T)$ for the spectrum if the space is obvious from the context. Recall that $\sigma (T)\subset \mathbb{C}$ is always a non-empty compact set. The \textit{point spectrum}, i.e. the set of eigenvalues of $T$, is denoted by $\sigma_p (T)$. The \textit{approximative point spectrum} of $T$ is defined by
$$ \sigma_{a} (T)=\{ \lambda \in \mathbb{C}: \exists \textnormal{ a sequence } (h_n) \subset \mathcal{H} \textnormal{ of unit vectors such that } \Vert (T-\lambda)h_n\Vert \longrightarrow 0\}.$$
Obviously, $\sigma_p (T)\subseteq \sigma_{a} (T)$. The following inclusions always hold

\begin{equation}\label{spincl}
\partial \sigma (T) \subseteq \sigma_{a} (T)\subseteq\sigma (T),
\end{equation}
where $\partial \sigma (T) $ denotes the boundary of $ \sigma (T)$.

If the operators $T: \mathcal{H}_1 \longrightarrow  \mathcal{H}_1$ an $V: \mathcal{H}_2 \longrightarrow  \mathcal{H}_2$ are \textit{similar}, i.e. there exists an invertible operator $U: \mathcal{H}_1 \longrightarrow  \mathcal{H}_2$ such that $T= U^{-1}VU$, then their spectra, point spectra and approximative point spectra coincide. This is very useful for our purposes since if we can write $\tilde{\tau}=g\circ \tau \circ g^{-1}$, where $g$ is an automorphism of $\Pi^+$ fixing infinity (so that $C_g$ is bounded and invertible on $\mathcal{A}_{\alpha}^2 ( \Pi^+)$ for $\alpha \geq -1$ and $ C_g^{-1}=C_{g^{-1}}$), then 
$$ C_{\tilde{\tau}}= C_g^{-1} C_{\tau} C_{g}.$$
This means that $C_{\tilde{\tau}}$ and $C_{\tau}$ both acting on $\mathcal{A}_{\alpha}^2 ( \Pi^+)$ for $\alpha \geq -1$ are similar and therefore $\sigma (C_{\tilde{\tau}} ) =\sigma (C_{\tau})$. Note that the operators $C_{\tau}: \mathcal{A}_{\alpha}^2 ( \Pi^+) \longrightarrow \mathcal{A}_{\alpha}^2 ( \Pi^+)$ and $C_{\varphi}: \mathcal{A}_{\alpha}^2 (\mathbb{D}) \longrightarrow \mathcal{A}_{\alpha}^2 (\mathbb{D})$, where $\alpha \geq -1$ and $\varphi  = g^{-1}\circ \tau \circ g $ for an analytic bijection $g: \mathbb{D} \longrightarrow \Pi^+$, are not similar in general. This is due to the fact that the composition operator $ C_{g}: \mathcal{A}_{\alpha}^2 ( \Pi^+)  \longrightarrow \mathcal{A}_{\alpha}^2 (\mathbb{D})  $ is not bounded below. However, $C_{\varphi}$ on $\mathcal{A}_{\alpha}^2 (\mathbb{D}) $ is similar to a weighted composition operator $MC_{\tau}$ on $ \mathcal{A}_{\alpha}^2 (\Pi^+)$, where the multiplication operator $M$ is in some cases simple enough to allow one to deduce the spectrum of $C_{\tau}$ from that of $C_{\varphi}$.

We are also able to determine the \textit{essential spectrum} of $C_{\tau}$ in all cases considered. The essential spectrum of $T$ is the closed set defined by
$$\sigma_{ess} (T; \mathcal{H}) =  \{\lambda \in \mathbb{C}: T-\lambda \textnormal{ is not a Fredholm operator on }\mathcal{H} \} .$$
Recall that, by Atkinson's theorem, the Fredholmness of $T$ on $\mathcal{H}$ is equivalent to the invertibility of $ (T-\lambda ) + \mathcal{K}(\mathcal{H})$ in the Calkin algebra $ \mathcal{L}(\mathcal{H})/ \mathcal{K}(\mathcal{H})$. Here and in the sequel $ \mathcal{L}(\mathcal{H})$ denotes the set of bounded linear operators on $\mathcal{H}$ and $\mathcal{K}(\mathcal{H})$ the set of compact operators on $\mathcal{H}$.

We will use the fact that the eigenvalues of infinite multiplicity are contained in the essential spectrum and that the essential spectra of similar operators coincide.

We denote by $\rho (T)$ the spectral radius of $T$, which is defined by $\rho (T):= \max \{\vert\lambda\vert : \lambda \in \sigma (T) \}$ and satisfies the spectral radius formula
\begin{equation}\label{spradfor}
\rho (T)= \lim_{n\longrightarrow \infty} \Vert T^n\Vert^{1/n}.
\end{equation}

\vspace{5pt}

Concerning the bounded composition operators acting on $\mathcal{A}_{\alpha}^2 ( \Pi^+)$ for $\alpha \geq -1$, the spectral radius is obtained from the angular derivative of the inducing map: The results by Elliott, Jury and Wynn (see \cite[Thm. 3.4]{EJ} and \cite[Thm. 3.5]{EW}) guarantee that whenever $C_{\tau}$ is bounded on $\mathcal{A}_{\alpha}^2 ( \Pi^+)$ for $\alpha \geq -1$, the spectral radius as well as the operator norm and the essential operator norm of $C_{\tau}$ satisfy 
\begin{equation}\label{sprad}
\rho (C_{\tau}) = \Vert C_{\tau}\Vert =\Vert C_{\tau}\Vert_{ess} =\big( \tau' (\infty)\big)^{(\alpha + 2)/2} .
\end{equation} 
Here, the essential operator norm is defined for $T:\mathcal{H} \longrightarrow \mathcal{H}$ by $\Vert T \Vert_{ess} = \inf \{ \Vert T -K\Vert : K \textnormal{ is a compact operator on }\mathcal{H} \}$. Moreover, replacing $\rho (T)$ by the essential spectral radius $ \rho_{ess} (T):= \max \{\vert\lambda\vert : \lambda \in \sigma_{ess} (T) \}$ and $ \Vert \cdot\Vert$ by $ \Vert \cdot\Vert_{ess}$ in Equation \eqref{spradfor} we have the essential spectral radius formula.

\section{The spectrum of $C_{\tau}$ when $\tau$ is a parabolic self-map of $\Pi^+$}\label{secpara}

In this section we consider the composition operator $C_{\tau}: \mathcal{A}_{\alpha}^2 ( \Pi^+) \longrightarrow \mathcal{A}_{\alpha}^2 ( \Pi^+)$ for $\alpha \geq -1$, induced by a parabolic self-map of $\Pi^{+}$, that is
$$\tau (w)= w+w_0, \quad \textnormal{where } \textnormal{Im}\, w_0 \geq 0, \, w_0\neq 0.$$
Recall that the only fixed point of $\tau$ is $\infty$ and $\tau' (\infty)=1$. Moreover, $\tau$ is an automorphism of $ \Pi^+$ if and only if $\textnormal{Im}\, w_0 =0$ ($w_0\neq 0$), in which case $C_{\tau}$ is invertible. 

\vspace{5pt}

Even though we cannot find a composition operator $ C_{g}: \mathcal{A}_{\alpha}^2 ( \Pi^+)  \longrightarrow \mathcal{A}_{\alpha}^2 (\mathbb{D}) $ (see Remark \ref{notsimi}) that would give a similarity between the operators $C_{\tau}: \mathcal{A}_{\alpha}^2 ( \Pi^+) \longrightarrow \mathcal{A}_{\alpha}^2 ( \Pi^+) $ and $C_{\varphi} : \mathcal{A}_{\alpha}^2 (\mathbb{D}) \longrightarrow \mathcal{A}_{\alpha}^2 (\mathbb{D})$, where $\varphi  = g^{-1}\circ \tau \circ g $, it turns out that their spectra are the same. (For the spectra of parabolic composition operators acting on $\mathcal{A}_{\alpha}^2 ( \mathbb{D})$ for $\alpha \geq -1$, see \cite[Thm. 7.5 and Cor. 7.42]{CM}.) The (essential) spectra of invertible ($\textnormal{Im}\, w_0 =0$) and self-adjoint (corresponding the case $\textnormal{Re}\, w_0 =0$ in the upper half-plane) parabolic composition operators on the Hardy space of the half-plane have been computed in \cite[Thm. 2.7 and Thm. 3.1]{Ma4}. Our proof is very different and works on $\mathcal{A}_{\alpha}^2 ( \Pi^+) $, for all $\alpha >-1$, as well. We are also able to show that the essential spectrum coincides with the spectrum in all cases.

\vspace{5pt}

\begin{thm}\label{sppara}
Let $\tau : \Pi^+ \longrightarrow \Pi^+$ be a parabolic map of the form $\tau (w)= w+w_0$, where $\textnormal{Im}\, w_0 \geq 0$ and $w_0\neq 0$. Then the spectrum of $C_{\tau}$ acting on $\mathcal{A}_{\alpha}^2 ( \Pi^+) $, for all $\alpha \geq -1$, is

\vspace{5pt}

\begin{itemize}
\item[i)] $ \sigma \big( C_{\tau}\big)= \mathbb{T}$, when $w_0\in \mathbb{R}$,

\vspace{5pt}

\item[ii)] $\sigma \big( C_{\tau}\big) = \{e^{iw_0t} : t\in [0,\infty ) \} \cup \{0\}$, when $w_0\in \Pi^+$.
\end{itemize}

\vspace{5pt}

\noindent Moreover, the essential spectrum of $C_{\tau}$ coincides with its spectrum in both cases.
\end{thm}

\begin{proof}
Denote 
$$ \mathcal{S}:= \overline{ \{ e^{iw_0t} : t\in [0,\infty ) \} } = \left\{ \begin{array}{ll}
\mathbb{T}, & \textnormal{ when } w_0\in \mathbb{R}, \\
\{ e^{iw_0t} : t\in [0,\infty ) \} \cup \{ 0\}, & \textnormal{ when } \textnormal{Im}\, w_0 >0.
\end{array}
\right.
$$

We will split the proof into two parts. In \textit{Step 1} we will prove that $ \sigma \big( C_{\tau}; \mathcal{A}_{\alpha}^2 ( \Pi^+) \big) \subseteq \mathcal{S}$ for all $\alpha \geq -1$ and in \textit{Step 2} we show that each point in $\mathcal{S} $ belongs to the essential spectrum of $  C_{\tau}$ on $ \mathcal{A}_{\alpha}^2 ( \Pi^+)$ for all $\alpha \geq -1$.

\vspace{5pt}

\textit{Step 1.} Note first that $ \big( \mathcal{F}^{-1}( C_{\tau} G) \big)(t) = e^{iw_0t} (\mathcal{F}^{-1}G) (t)$ for any $G\in \mathcal{A}_{\alpha}^2 ( \Pi^+) $ and $t>0$. Therefore, by the Paley-Wiener theorem (see p. \pageref{pwnormeq}), the Fourier transform gives us a similarity between the operators $C_{\tau} : \mathcal{A}_{\alpha}^2 ( \Pi^+) \longrightarrow \mathcal{A}_{\alpha}^2 ( \Pi^+)$ and $\widehat{C}_{\tau} : L_{\alpha +1}^2 \longrightarrow L_{\alpha +1}^2$, where
\begin{equation}\label{ceetausim}\widehat{C}_{\tau}f(t)=e^{iw_0t}f(t).
\end{equation}

Since $\widehat{C}_{\tau} $ is a multiplication operator on $L_{\alpha +1}^2$, we can easily find points $\lambda \in \mathbb{C}$ for which $ \widehat{C}_{\tau} -\lambda$ \textit{is} invertible. Let $\lambda \in \mathbb{C}\setminus \mathcal{S}$. Since $\mathbb{C}\setminus \mathcal{S}$ is an open set, $\textnormal{dist}\, (\lambda, \mathcal{S})=:\delta >0$. For all $t> 0$, define a function $k$ by $k(t)=\frac{1}{e^{iw_0t}-\lambda}$. Since $ \vert k(t)\vert \leq 1/\delta$, the multiplication operator $M_k$, where $M_kf =kf$, is bounded on $ L_{\alpha +1}^2 $. Now,
$$M_k(\widehat{C}_{\tau} - \lambda)f =k (e^{iw_0t}-\lambda )f = f= (e^{iw_0t}-\lambda )kf = (\widehat{C}_{\tau} - \lambda )M_k f,$$
which means that $M_k$ is the inverse of $\widehat{C}_{\tau} - \lambda $ and therefore $\lambda \in \mathbb{C}\setminus \sigma \big( \widehat{C}_{\tau}\big)$. It follows that $ \sigma \big( C_{\tau}; \mathcal{A}_{\alpha}^2 ( \Pi^+) \big) = \sigma \big( \widehat{C}_{\tau} ; L_{\alpha +1}^2\big)\subseteq \mathcal{S} $.

\vspace{5pt}

\textit{Step 2.} In order to show that $\mathcal{S}\subseteq \sigma_{ess} \big( C_{\tau}; \mathcal{A}_{\alpha}^2 ( \Pi^+) \big)$ for all $\alpha \geq -1$ we need the following fact (see \cite[Thm. XI.2.3]{Con}): \textit{Let $T:\mathcal{H} \longrightarrow \mathcal{H} $ be a bounded linear operator on a Hilbert space $\mathcal{H}$. If there exists an orthonormal sequence $(f_n)\subset \mathcal{H}$ such that}
$$
\Vert (T-\lambda) f_n\Vert \longrightarrow 0, \textnormal{ as } n\longrightarrow \infty ,
$$
\textit{then} $\lambda \in \sigma_{ess} \big(T; \mathcal{H}\big)$.

\vspace{5pt}

As in \textit{Step 1} we will use the similarity provided by the Paley-Wiener theorem. Now our aim is to find an orthogonal sequence of functions $(g_n)$ in $L_{\alpha +1}^2$ for any $\alpha \geq -1$ and $\lambda \in \mathcal{S}$,
$$\frac{\Vert (\widehat{C}_{\tau} -\lambda)g_n \Vert_{L_{\alpha +1}^2}}{\Vert g_n \Vert_{L_{\alpha +1}^2 }} \longrightarrow 0,\textnormal{ as } n \longrightarrow \infty.$$ 

Fix $t_0 \geq 0$ and $w_0\in \Pi^+ \cup \mathbb{R}$. Let $\lambda =e^{iw_0t_0}\in\mathcal{S}$ and $(\epsilon_n)\subset \mathbb{R}$ be any sequence satisfying $\epsilon_n >\epsilon_{n+1} $ and $\epsilon_n \longrightarrow 0$, as $n\longrightarrow \infty$ (for instance, set $\epsilon_n =1/n$). Consider the orthogonal sequence $(g_n)\subset L_{\alpha +1}^2$ of characteristic functions 
$$ g_n= \chi_{\big[t_0 + \epsilon_{n+1} , t_0 + \epsilon_{n} ]} $$
for which it holds that
$$(\widehat{C}_{\tau} -\lambda ) g_n= (e^{iw_0t} -e^{iw_0t_0}) g_n .$$ 
Now we have that
\begin{equation*}
\frac{\Vert (\widehat{C}_{\tau} -\lambda ) g_n \Vert_{L_{\alpha +1}^2}^2}{\Vert g_n \Vert_{L_{\alpha +1}^2}^2}
\leq \sup_{t_0 +\epsilon_{n+1} \leq t \leq t_0 +\epsilon_{n}} \vert e^{iw_0t} - e^{iw_0t_0}\vert^2 \longrightarrow 0, \textnormal{ as } n \longrightarrow \infty.
\end{equation*} 

It follows that $\mathcal{S} \subseteq \sigma_{ess} \big(\widehat{C}_{\tau}; L_{\alpha +1}^2\big)= \sigma_{ess} \big( C_{\tau}; A_{\alpha }^2 (\Pi^+)\big)$ for all $\alpha \geq -1$. 

\end{proof}

\section{The spectrum of $C_{\tau}$ when $\tau$ is a hyperbolic self-map of $\Pi^+$}\label{hypsec}

We will first consider the case where $\tau$ is a hyperbolic automorphism of $\Pi^+$, that is, $\tau$ is of the form $\tau (w)=\mu w + w_0$, where $\mu\in (0,1)\cup (1,\infty)$ and $w_0\in \mathbb{R}$. Any mapping of this form can be written as 
$$\tau = g \circ \tilde{\tau} \circ g^{-1},$$
where $ \tilde{\tau} (w)=\mu w$ and $g(w)= w+ \frac{w_0}{1-\mu}$. Since $g$ is an automorphism of $\Pi^+$ fixing $\infty$, it induces an invertible composition operator $C_g$ on $\mathcal{A}_{\alpha}^2 ( \Pi^+) $ so that $ C_{\tau}$ and $C_{ \tilde{\tau}} $ acting on $\mathcal{A}_{\alpha}^2 ( \Pi^+)$, for $\alpha \geq -1$, are similar. Therefore, in order to compute the spectrum for $C_{\tau}$ it is enough to compute the spectrum for the composition operator induced by the mapping
$$w \mapsto \mu w,\quad \mu\in (0,1)\cup (1,\infty).$$
We will first compute the spectrum and the essential spectrum in the case $\mu \in (0,1)$ (Theorem \ref{sphypauto}) and from this we will also get the case $\mu \in (1,\infty)$ (Corollary \ref{sphypautoinverse}). The spectrum and the essential spectrum of the invertible hyperbolic composition operators in the Hardy space of the (right) half-plane $\mathbb{C}_+$ have been computed by Matache in \cite[Thm. 2.12]{Ma4}. Matache's proof for the spectrum is based on the similarity between $C_{\tau}$ on $H^2 (\mathbb{C}_+)$ and a certain weighted composition operator on $H^2 (\mathbb{D})$. The spectrum is then obtained from the results in \cite[Thm. 4.8]{HLNS}. It it possible that this idea could also be used in the weighted Bergman spaces. Our approach is different and more direct revealing simultaneously also the essential spectrum on all of the spaces.

\vspace{5pt}

It is worth noting that the result in this case differs essentially from the unit disc setting, where, in the classical Hardy and the weighted Bergman spaces, the spectrum of an invertible hyperbolic composition operator is always an annulus centred at the origin (see \cite[Thm. 7.4]{CM}). In the half-plane setting we get that the spectrum is always a circle centred at the origin with radius depending on the space.

\vspace{5pt}

\begin{thm}\label{sphypauto}
Let $\tau$ be a hyperbolic automorphism of $\Pi^+$ of the form $\tau (w)=\mu w $, where $ \mu \in (0,1)$. Then the spectrum and the essential spectrum of $C_{\tau}$ acting on $ \mathcal{A}_{\alpha}^2 ( \Pi^+)$, for $\alpha \geq -1$, are
$$ \sigma \big( C_{\tau} ; \mathcal{A}_{\alpha}^2 ( \Pi^+)\big) = \sigma_{ess} \big( C_{\tau} ; \mathcal{A}_{\alpha}^2 ( \Pi^+)\big)= \big\{ \lambda  \in \mathbb{C} : \vert \lambda \vert = \mu^{-(\alpha +2)/2}\big\} .$$

\end{thm}

\begin{proof}
The fixed points of $\tau$ are $0$ (attractive) and $\infty$. Moreover, by \eqref{attrepder}
$$\tau' (\infty) = 1/\tau' (0)=\mu^{-1},$$
so that the spectral radius of $ C_{\tau}$ acting on $ \mathcal{A}_{\alpha}^2 ( \Pi^+)$ is $\mu^{-(\alpha +2)/2} $ (see \eqref{sprad}). We will split the proof into two parts: in \textit{Step 1} we show that $\sigma \big( C_{\tau} ; \mathcal{A}_{\alpha}^2 ( \Pi^+)\big) \subseteq \big\{ \lambda  \in \mathbb{C} : \vert \lambda \vert = \mu^{-(\alpha +2)/2}\big\} =: \mathcal{Y}_{\alpha}$ and in \textit{Step 2} we prove that each point in $\mathcal{Y}_{\alpha} $ belongs to the essential spectrum of $C_{\tau} $ on $\mathcal{A}_{\alpha}^2 ( \Pi^+)$.

\vspace{5pt}

\textit{Step 1.} Let $F \in \mathcal{A}_{\alpha}^2 ( \Pi^+)$, $\alpha >-1$, and $\lambda \in \mathbb{C}$. Now, by the change of variable $ w \longrightarrow \mu^{-1}w$, we get that (for $w=x+iy$)
\begin{equation*}
\begin{split}
\Vert C_{\tau} F\Vert_{\mathcal{A}_{\alpha}^2 ( \Pi^+)}^2 &= \int_{\Pi^+} \vert F(\mu w)\vert^2 y^{\alpha} \, dx \, dy = \int_{\Pi^+} \vert F( w)\vert^2 y^{\alpha} \mu^{-\alpha -2}\, dx \, dy \\
&=  \mu^{-\alpha -2} \Vert F\Vert_{\mathcal{A}_{\alpha}^2 ( \Pi^+)}^2.
\end{split}
\end{equation*}
For $H^2 (\Pi^+)$ ($\alpha =-1$), we get similarly that
$$\Vert C_{\tau} F\Vert_{H^2 (\Pi^+)}^2 =\sup_{\mu y>0} \int_{-\infty}^{\infty} \vert F(\mu x+i\mu y)\vert^2 \, dx =\sup_{y>0} \int_{-\infty}^{\infty} \vert F(x+iy)\vert^2 \mu^{-1} \, dx = \mu^{-1} \Vert  F\Vert_{H^2 (\Pi^+)}^2.$$
Since $\Vert C_{\tau} F\Vert_{\mathcal{A}_{\alpha}^2 ( \Pi^+)} =\mu^{-(\alpha +2)/2} \Vert F\Vert_{\mathcal{A}_{\alpha}^2 ( \Pi^+)}$ for all $\alpha \geq -1$ and the spectrum of an isometric operator is contained in the unit circle, we are able to deduce that
$$\sigma \big( C_{\tau} ; \mathcal{A}_{\alpha}^2 ( \Pi^+)\big) \subseteq \mathcal{Y}_{\alpha}.$$

\vspace{5pt}

\textit{Step 2.} Here we use a similar technique as in \textit{Step 2} on p. \pageref{sppara}. Since $ \mathcal{F}^{-1}\big(G (\mu x)\big)= \frac{1}{\mu}( \mathcal{F}^{-1}G )(x/\mu) $ for any $G\in \mathcal{A}_{\alpha}^2 ( \Pi^+)$ and $x>0$, we have, by the (versions of) Paley-Wiener theorem, that the operator $  \widehat{C}_{\tau} : L_{\alpha +1}^2\longrightarrow  L_{\alpha +1}^2$ defined by
$$ \widehat{C}_{\tau} f (x) = \frac{1}{\mu} f(x/\mu) , \quad x>0,$$
is similar to $C_{\tau}: \mathcal{A}_{\alpha}^2 ( \Pi^+)\longrightarrow \mathcal{A}_{\alpha}^2 ( \Pi^+)$, for all $\alpha \geq -1$. For each $\lambda\in  \mathcal{Y}_{\alpha}$, we find an orthogonal sequence $(g_n)\subset L_{\alpha +1}^2$ such that 
$$ \frac{\Vert ( \widehat{C}_{\tau} - \lambda )g_n \Vert_{L_{\alpha +1}^2}}{\Vert g_n\Vert_{L_{\alpha +1}^2}} \longrightarrow 0, \textnormal{ as }n\longrightarrow\infty .$$

Fix $\lambda_{\alpha} = \mu^{-(\alpha +2)/2}e^{it}$, where $t\in \mathbb{R}$, and define (cf. \cite[Thm. 3.2]{Hi} in the Dirichlet space case)
$$g(x)= x^{\alpha / 2} \mu^{-(\alpha +2)/2 }e^{-it \log_{\mu} x} , \quad x>0.$$
By inspection,
\begin{equation}\label{geemyy}
\frac{1}{\mu }g(x/ \mu) = \mu^{-(\alpha +2)/2}e^{it} g(x) = \lambda_{\alpha} g(x).
\end{equation}

Let $a_0=1$ and $a_n = e^{-n}a_{n-1}=e^{-1-2- \cdots -n}$ for all $n\geq 1$. Note that $ a_n < a_{n-1}$ for all $n\in \mathbb{N}$ and that $a_n \longrightarrow 0$, as $n\longrightarrow\infty$. Define a sequence $(g_n)\subset L_{\alpha +1}^2$ by setting
$$ g_n = g\cdot \chi_{[a_n, a_{n-1}]}. $$
Since the sequence $(a_n)$ is decreasing, $ \langle g_n ,g_m \rangle_{L_{\alpha +1}^2} =0$ for all $n \neq m$. Moreover,
\begin{equation}\label{normg}
\begin{split}
 \Vert g_n \Vert_{L_{\alpha +1}^2}^2 &= b_{\alpha}\int_{a_n}^{a_{n-1}} \big\vert x^{\alpha / 2} \mu^{-(\alpha +2)/2 }e^{-it \log_{\mu} x} \big\vert^2 x^{-(\alpha +1)}\, dx = b_{\alpha} \mu^{-(\alpha +2)} \int_{a_n}^{a_{n-1}} x^{-1}\, dx \\
 &= b_{\alpha} \mu^{-(\alpha +2)} n.
\end{split}
\end{equation}

On the other hand, by using \eqref{geemyy}, we get (for $n$ big enough)
\begin{equation}\label{normctau}
\begin{split}
 \Vert  ( \widehat{C}_{\tau} - \lambda_{\alpha})g_n \Vert_{L_{\alpha +1}^2}^2 &= b_{\alpha} \int_0^{\infty} \big\vert \lambda_{\alpha} g(x)\chi_{[\mu a_n ,\mu a_{n-1} ]} -\lambda_{\alpha} g(x)\chi_{ [a_n ,a_{n-1} ]} \big\vert^2   x^{-(\alpha +1)} \, dx   \\
 &\stackrel{(*)}{=} b_{\alpha} \vert \lambda_{\alpha} \vert^2 \Big( \int_{\mu a_n}^{a_n} \vert g(x)\vert^2 x^{-(\alpha +1)}\, dx + \int_{\mu a_{n-1}}^{a_{n-1}} \vert g(x)\vert^2 x^{-(\alpha +1)}\, dx\Big)  \\
 &=  b_{\alpha}\vert \lambda_{\alpha} \vert^2 \mu^{-(\alpha +2)} \Big( \int_{\mu a_n}^{a_n} x^{-1}\, dx + \int_{\mu a_{n-1}}^{a_{n-1}} x^{-1}\, dx\Big)  \\
 &= -2b_{\alpha}\vert \lambda_{\alpha} \vert^2 \mu^{-(\alpha +2)} \ln \mu .
 \end{split}
\end{equation}
In $(*)$ we have used the fact that $a_n <\mu a_{n-1}$ for all $n > \ln (1/ \mu )$ and so the function $ g(x) \chi_{[\mu a_n ,\mu a_{n-1}]}(x)  -  g(x)\chi_{ [a_n ,a_{n-1}]}$ vanishes on $[0,\mu a_n]\cup [ a_n, \mu a_{n-1}]$.

By Equations \eqref{normg} and \eqref{normctau} we have that
$$ \frac{\Vert ( \widehat{C}_{\tau} -\lambda_{\alpha} ) g_n\Vert_{L_{\alpha +1}^2}^2}{\Vert  g_n \Vert_{L_{\alpha +1}^2}^2} = \frac{ -2\vert \lambda_{\alpha} \vert^2  \ln \mu}{n} \longrightarrow 0, \textnormal{ as } n\longrightarrow \infty  .$$
Now \cite[Thm. XI.2.3]{Con} guarantees that $\lambda_{\alpha} \in \sigma_{ess} \big( \widehat{C}_{\tau} ; L_{\alpha +1}^2\big)$. Furthermore, by similarity, this yields
$$\mathcal{Y}_{\alpha}:= \{\lambda \in \mathbb{C} :  \vert \lambda\vert = \mu^{-(\alpha +2)/2}\}\subseteq \sigma_{ess} \big( C_{\tau}; \mathcal{A}_{\alpha}^2 ( \Pi^+)\big).$$

\end{proof}

From the (essential) spectrum of $C_{\tau}$ we will obtain the (essential) spectrum of the inverse of $C_{\tau}$ by using general spectral theory.

\begin{cor}\label{sphypautoinverse}
Let $\tau^{-1} (w)= \mu^{-1} w$, where $\mu\in (0,1)$. The spectrum of $C_{\tau^{-1}}$ acting on $ \mathcal{A}_{\alpha}^2 ( \Pi^+)$, $\alpha \geq -1$, is
$$ \sigma \big( C_{\tau^{-1}} ; \mathcal{A}_{\alpha}^2 ( \Pi^+)\big) =\sigma_{ess} \big( C_{\tau^{-1}} ; \mathcal{A}_{\alpha}^2 ( \Pi^+)\big)= \big\{ \lambda  \in \mathbb{C} : \vert \lambda \vert = \mu^{(\alpha +2)/2}\big\} .$$
\end{cor}
\begin{proof}
Since $C_{\tau^{-1}}= C_{\tau}^{-1}$, we have that $\lambda \in \sigma \big( C_{\tau^{-1}} ; \mathcal{A}_{\alpha}^2 ( \Pi^+)\big)$ if and only if $ \lambda^{-1}\in \sigma \big( C_{\tau} ; \mathcal{A}_{\alpha}^2 ( \Pi^+)\big)$, where $\tau (w)=\mu w$. The same is true also for the essential spectrum and so the result follows from Theorem \ref{sphypauto}.
\end{proof}

\vspace{5pt}

Let us consider then the cases where $\tau$ is a hyperbolic self-map of $\Pi^+$ and $\tau (\Pi^+)\subsetneq \Pi^+$. Let $\mu \in (0,1)$ so that $ \mu^{-1}\in (1,\infty)$, and consider first the mapping 
$$\tau (w)= \frac{w}{\mu} +w_0 \textnormal{, where } \textnormal{Im}\, w_0 >0.$$ The attractive fixed point of $\tau$ is $\infty$ and the repulsive fixed point will necessarily have negative imaginary part (just solve $\tau (w)=w$ to see this). Moreover, we can write
$$ \tau_1 = g_1^{-1}\circ \tau \circ g_1, $$
where 
$$\tau_1 (w)= \frac{w}{\mu} + \frac{i(1-\mu )}{\mu} $$
and $g_1(w) = \mu (1-\mu )^{-1} \big( (\textnormal{Im}\, w_0 )w - \textnormal{Re}\, w_0 \big)$. Since $\frac{\mu \textnormal{Im}\, w_0}{1-\mu} >0$ and $\frac{ \mu  \textnormal{Re}\, w_0}{\mu -1}\in \mathbb{R}$, $g_1$ is an automorphism of $ \Pi^+$ fixing $\infty$ and, therefore, induces a bounded invertible composition operator $ C_{g_1}$ on $\mathcal{A}_{\alpha}^2 ( \Pi^+)$ for all $\alpha \geq -1$. It follows that $ C_{\tau}$ and $C_{\tau_1}$ acting on $\mathcal{A}_{\alpha}^2 ( \Pi^+)$ for $\alpha \geq -1$ are similar operators and in order to compute the spectrum for $C_{\tau}$ it is enough to compute the spectrum for $C_{\tau_1}$. In the proof below, recall our convention $\mathcal{A}_{-1}^2 (\mathbb{D})$ for $H^2 (\mathbb{D})$.

\vspace{5pt}

\begin{thm}\label{hyptauyks}
Let $\mu \in (0,1)$ and $\tau_1$ be a hyperbolic self-map of $\Pi^+$ of the form $\tau_1 (w)= \frac{w}{\mu} + \frac{i(1-\mu )}{\mu} $. Then, for all $\alpha \geq -1$, the spectrum and the essential spectrum are
$$ \sigma \big(C_{\tau_1} ; \mathcal{A}_{\alpha}^2 ( \Pi^+)\big) = \sigma_{ess} \big(C_{\tau_1} ; \mathcal{A}_{\alpha}^2 ( \Pi^+)\big) =\{\lambda  \in \mathbb{C} : \vert \lambda\vert \leq \mu^{(\alpha +2)/2}\}.$$
\end{thm}

\begin{proof} 
The fixed points of $\tau_1$ are $-i$ and $\infty$ (the attractive one). By \eqref{sprad} and \eqref{attrepder}, for $\alpha \geq -1$,
$$\rho \big(  C_{\tau_1};   \mathcal{A}_{\alpha}^2 ( \Pi^+) \big) = \big(\tau_1' (\infty)\big)^{(\alpha +2)/ 2}=\big(\tau_1' (-i)\big)^{-(\alpha +2)/ 2}  = \mu^{(\alpha +2)/ 2},$$
so that
$$\sigma \big(C_{\tau_1}; \mathcal{A}_{\alpha}^2 ( \Pi^+) \big)\subseteq \{\lambda  \in \mathbb{C}: \vert \lambda\vert \leq \mu^{(\alpha +2)/2}\}.$$
We will show that each point in the open disc $\{\lambda  \in \mathbb{C} : \vert \lambda\vert < \mu^{(\alpha +2)/2}\}$ is an eigenvalue of infinite multiplicity for $C_{\tau_1}$ on $\mathcal{A}_{\alpha}^2 ( \Pi^+)  $ for $\alpha \geq -1 $. Since the eigenvalues of infinite multiplicity belong to the essential spectrum, which is always a closed set, the claim will then follow.

We will take advantage of the known spectrum in the corresponding case of the unit disc. Namely, by Hurst \cite[Theorem 8]{Hu}, we know that the spectrum of $C_{\varphi_1}$, where $\varphi_1 (z)=\mu z +1 -\mu$, acting on $\mathcal{A}_{\alpha}^2 (\mathbb{D})$ for $\alpha \geq -1$, is the closed disc with radius $\varphi_1' (1)^{-(\alpha  +2)/2} = \mu^{-(\alpha  +2)/2}$. Moreover, each point $\lambda $ satisfying $ \vert \lambda \vert <  \mu^{-(\alpha  +2)/2}$ is an eigenvalue of infinite multiplicity for $C_{\varphi_1} : \mathcal{A}_{\alpha}^2 (\mathbb{D})\longrightarrow \mathcal{A}_{\alpha}^2 (\mathbb{D})$.

By a simple computation, we see that $\varphi_1 = h^{-1} \circ \tau_1 \circ h$, where $h: \mathbb{D} \longrightarrow \Pi^+$ is the bijection defined in \eqref{hoo}. Using the isometric isomorphism $J: \mathcal{A}_{\alpha}^2 (\mathbb{D})\longrightarrow \mathcal{A}_{\alpha}^2 ( \Pi^+)$ (see \eqref{jii}) we get that
\begin{equation*}
\begin{split}
J(f\circ \varphi_1)(w)& = (f\circ \varphi_1 \circ h^{-1})(w)\frac{2^{\alpha +1}}{(w+i)^{\alpha +2}} = (f\circ h^{-1} \circ \tau_1 )(w)\frac{2^{\alpha +1}}{(\tau_1 (w)+i)^{\alpha +2}}\Big(\frac{\tau_1 (w)+i}{w+i}\Big)^{\alpha +2} \\
&= \Big(\frac{\tau_1 (w)+i}{w+i}\Big)^{\alpha +2} \big(C_{\tau_1} (Jf)\big) (w)  .
\end{split}
\end{equation*}
On the other hand, observe that
$$ \frac{\tau_1 (w)+i}{w+i} = \frac{\frac{w}{\mu} + \frac{i(1-\mu )}{\mu}+i}{w+i}= \mu^{-1}\frac{w+ i -i\mu +i\mu}{w+i}=  \mu^{-1},$$
so that
$$\big( J(C_{ \varphi_1}f)\big) (w)= \mu^{-(\alpha +2)}  \big(C_{\tau_1} (Jf)\big) (w).$$

That is, $C_{\varphi_1}$ on $ \mathcal{A}_{\alpha}^2 (\mathbb{D})$ is similar to the operator $\mu^{-(\alpha +2)}C_{\tau_1}$ acting on $\mathcal{A}_{\alpha}^2 ( \Pi^+) $. 

\vspace{5pt}

Recall that, by \cite[Theorem 8]{Hu}, for any $\lambda$ with $ \vert\lambda\vert <  \mu^{-(\alpha  +2)/2} $ there exists $f_{\lambda}\in \mathcal{A}_{\alpha}^2 (\mathbb{D})$ (in fact, an $\infty$-dimensional subspace of them) such that $C_{\varphi_1} f_{\lambda} = \lambda f_{\lambda}$. But now, writing $F_{\lambda} =Jf_{\lambda}$, we have that
$$ \mu^{-(\alpha +2)} C_{\tau_1}F_{\lambda}= \lambda F_{\lambda} $$
which is equivalent to
$$  C_{\tau_1}F_{\lambda} = \mu^{\alpha +2} \lambda F_{\lambda} .$$
Since $ \vert\lambda\vert <  \mu^{-(\alpha  +2)/2} $, we obtain that $C_{\tau_1}$ has an eigenfunction for each point in the open disc of radius $\mu^{(\alpha +2)/ 2}  $. Moreover, the eigenvalues are of infinite multiplicity since $J$ preserves the linear independence of functions. 

\vspace{5pt}

In fact, it is possible to compute the eigenfunctions in $ \mathcal{A}_{\alpha}^2 ( \Pi^+)$ explicitly from the corresponding functions in the unit disc setting: In the unit disc setting (see the proof of Theorem $8$ in \cite{Hu}), for any $p > -(\alpha +2)/2 $ and $q\in \mathbb{R}$, the function $f_{p,q} \in \mathcal{A}_{\alpha}^2 (\mathbb{D})$,
$$f_{p,q} (z)= \exp \{ (p+iq) \log (1-z)\} = (1-z)^{p+iq},$$
where $\log (\cdot )$ is the principal branch of the natural logarithm, is an eigenvector for $C_{\varphi_1}$ corresponding to the eigenvalue $\lambda = \mu^{p+iq}$. Using the isometric isomorphism we get that
$$(Jf_{p,q})(w) = \Big(1-\frac{w-i}{w+i}\Big)^{p+iq} \frac{c_{\alpha}}{ (w+i)^{\alpha +2}}= \frac{(2i)^{p+iq}c_{\alpha}}{(w+i)^{\alpha +2 + p+iq }}
$$
is an eigenvector for $C_{\tau_1}$ acting on $\mathcal{A}_{\alpha}^2 ( \Pi^+)$ corresponding to the eigenvalue $\mu^{\alpha + 2 + p+iq}$.

\end{proof}

\vspace{5pt}

Consider next the mapping $\tau  (w)= \mu w +w_0$, where $\mu \in (0,1)$ and $\textnormal{Im}\, w_0 >0$. Here, the repulsive fixed point of $\tau$ is $\infty$ and the attractive fixed point belongs to $\Pi^+$. In this case, to compute the spectrum for $C_{\tau}$, it is enough to consider the mapping $ \tau_2 : \Pi^+ \longrightarrow \Pi^+$ of the form
$$ \tau_2 (w)= \mu w+ i (1-\mu),$$
having $i$ as its attractive fixed point. This is due to the fact that we can write
$$\tau_2 = g_2^{-1}\circ \tau \circ g_2 , $$
where $g_2(w)= (1-\mu)^{-1} \big((\textnormal{Im}\, w_0)w + \textnormal{Re}\, w_0 \big) $ is an automorphism of $\Pi^+$ fixing $\infty$, so that $C_{g_2}$ is bounded and invertible on $ \mathcal{A}_{\alpha}^2 ( \Pi^+)$ for all $\alpha \geq -1$. From this it follows that
$$C_{\tau_2} = C_{g_2}C_{\tau}C_{g_2}^{-1}.$$
Therefore, for any $\alpha \geq -1$, we have that $\sigma \big(C_{\tau_2} ; \mathcal{A}_{\alpha}^2 ( \Pi^+) \big) =\sigma \big(C_{\tau} ; \mathcal{A}_{\alpha}^2 ( \Pi^+) \big)$ and that the essential spectra of $C_{\tau_2}$ and $C_{\tau}$ on $\mathcal{A}_{\alpha}^2 ( \Pi^+)$ coincide. The spectrum and the essential spectrum for $C_{\tau_2}  $ are determined in the following result.

\vspace{5pt}

\begin{thm}\label{hyptaukaks}
Let $\mu \in (0,1)$ and $\tau_2$ be a hyperbolic self-map of $\Pi^+$ of the form $\tau_2 (w)=\mu w+ i (1-\mu) $. Then, for all $\alpha \geq -1$, the spectrum and the essential spectrum of $C_{\tau_2}$ on $\mathcal{A}_{\alpha}^2 ( \Pi^+) $ are
$$ \sigma \big(C_{\tau_2} ; \mathcal{A}_{\alpha}^2 ( \Pi^+) \big) = \sigma_{ess} \big(C_{\tau_2} ; \mathcal{A}_{\alpha}^2 ( \Pi^+) \big) =\{\lambda  \in \mathbb{C} : \vert \lambda\vert \leq \mu^{-(\alpha +2)/2}\}.$$

\end{thm}

\begin{proof}
The mapping $\tau_2$ has $i$ (attractive) and $\infty$ as its fixed points. Again, by \eqref{sprad}, we know that the spectral radius of $C_{\tau_2}$ is $\big(\tau_2' (\infty)\big)^{(\alpha +2)/2}= \big(\tau_2' (i)\big)^{-(\alpha +2)/2} =\mu^{- (\alpha +2)/2}$ so that
$$\sigma \big(C_{\tau_2}; \mathcal{A}_{\alpha}^2 ( \Pi^+) \big)\subseteq \{\lambda  \in \mathbb{C} : \vert \lambda\vert \leq \mu^{-(\alpha +2)/2}\}.$$
To prove the reverse containment (and that the essential spectrum coincides with the spectrum) it is enough to show that each point in the disc $\{\lambda  \in \mathbb{C} : \vert \lambda\vert < \mu^{-(\alpha +2)/2}\}$ is an eigenvalue (of infinite multiplicity) for the adjoint  $C_{\tau_2}^{*}$. From this it follows that 
$$\{\lambda  \in \mathbb{C} : \vert \lambda\vert < \mu^{-(\alpha +2)/2}\} \subseteq \sigma \big(C_{\tau_2}; \mathcal{A}_{\alpha}^2 ( \Pi^+) \big) $$ 
since $\lambda \in \sigma  (C_{\tau_2})$ if and only if $\overline{\lambda } \in \sigma  ( C_{\tau_2}^{*}) $. Since the eigenvalues of infinite multiplicity belong to the essential spectrum and an operator is a Fredholm operator if and only if its adjoint is (see \cite[Thm. 16.4.]{Mu}, for instance), we will also get that
$$ \sigma_{ess} \big( C_{\tau_2}  ; \mathcal{A}_{\alpha}^2 ( \Pi^+) \big) = \sigma_{ess}\big( C_{\tau_2}^{*}  ; \mathcal{A}_{\alpha}^2 ( \Pi^+) \big) =  \sigma \big( C_{\tau_2}^{*}  ; \mathcal{A}_{\alpha}^2 ( \Pi^+) \big) =  \sigma \big( C_{\tau_2}  ; \mathcal{A}_{\alpha}^2 ( \Pi^+) \big).
$$ 

Our proof relies on a useful observation, namely, that in this case, $C_{\tau_2}^{*} = \mu^{-(\alpha +2)} C_{\tau_1}$, where $\tau_1(w)= \mu^{-1} w + i\mu^{-1}(1-\mu )$ as in Theorem \ref{hyptauyks}. The claim follows then from Theorem \ref{hyptauyks}.

Recall first that the reproducing kernel $K_z^{\alpha}$ of $\mathcal{A}_{\alpha}^2 ( \Pi^+)$ at the point $z\in \Pi^{+}$ is of the form $K_z^{\alpha} (w)= \frac{k_{\alpha}}{(w-\bar{z})^{\alpha +2}}$, where the constant $k_{\alpha}$ depends only on $\alpha$. Let $F\in \mathcal{A}_{\alpha}^2 ( \Pi^+)$. Now
$$ C_{\tau_2}^{*} F (z) = \langle C_{\tau_2}^{*} F , K_z^{\alpha} \rangle_{\mathcal{A}_{\alpha}^2 ( \Pi^+)} = \langle  F , C_{\tau_2} K_z^{\alpha} \rangle_{\mathcal{A}_{\alpha}^2 ( \Pi^+)} ,$$
where 
\begin{equation*}
\begin{split}
C_{\tau_2} K_z^{\alpha} (w) &=  \frac{k_{\alpha}}{(\tau_2 (w)-\bar{z})^{\alpha +2}} 
= \frac{k_{\alpha}}{(\mu w + i(1-\mu )-\bar{z})^{\alpha +2}} \\
& = \mu^{-(\alpha +2)} \frac{k_{\alpha}}{\Big( w - \overline{\big( \frac{z}{\mu} +\frac{i(1-\mu )}{\mu}\big)}\Big)^{\alpha +2}} = \mu^{-(\alpha +2)} K_{\tau_1 (z)}^{\alpha} (w).
\end{split}
\end{equation*}
Here, $\tau_1 $ is as in Theorem \ref{hyptauyks}. It follows that for all $z\in \Pi^{+} $
\begin{equation*}
\begin{split}
C_{\tau_2}^{*} F (z)  &=  \langle  F , C_{\tau_2} K_z^{\alpha} \rangle_{\mathcal{A}_{\alpha}^2 ( \Pi^+)} = \mu^{-(\alpha +2)} \langle  F , K_{\tau_1 (z)}^{\alpha}  \rangle_{\mathcal{A}_{\alpha}^2 ( \Pi^+)} \\
& = \mu^{-(\alpha +2)} F\big(\tau_1 (z) \big) =  \mu^{-(\alpha +2)} C_{\tau_1}F(z). 
\end{split}
\end{equation*}

\end{proof}

\vspace{5pt}

\begin{rem}\label{notsimi} Note that in the hyperbolic automorphism case $\sigma \big(C_{\tau}; \mathcal{A}_{\alpha}^2 ( \Pi^+)\big) \neq \sigma \big(C_{\varphi}; \mathcal{A}_{\alpha}^2 ( \mathbb{D})\big)$ for $\alpha \geq -1$, where $ \varphi = g^{-1}\circ \tau \circ g$ for any analytic bijection $g: \mathbb{D} \longrightarrow \Pi^+$. From this we can deduce that there does not exist a composition operator $C_g$ that would give an isomorphism between the spaces $\mathcal{A}_{\alpha}^2 ( \Pi^+) $ and $\mathcal{A}_{\alpha}^2 (\mathbb{D})$. The fact that in the non-automorphism case the operators $C_{\tau_1}$ and $C_{\tau_2}$ are (up to a multiplication by a scalar) adjoints of each other is a phenomenon that does not occur in the unit disc setting.

\end{rem}

\section{The spectrum of $C_{\tau}$ on $\mathcal{D}_{\alpha}^2 (\Pi^+) $ for $\alpha > -1$}\label{secdir}

From the spectral results for the parabolic and the hyperbolic composition operators in the Hardy space and the weighted Bergman spaces of the upper half-plane we will get the (essential) spectra in the weighted Dirichlet spaces of the upper half-plane in a straightforward manner. 

\vspace{5pt}

We will use the definition given in \cite[Section 6]{DGM} for the weighted Dirichlet spaces on $\Pi^+$: For all $\alpha >-1$, the space $\mathcal{D}_{\alpha}^2 (\Pi^+) $ consists of the analytic functions $ F: \Pi^+ \longrightarrow \mathbb{C}$ such that $F' \in \mathcal{A}_{\alpha}^2 ( \Pi^+) $, i.e.
\begin{equation}\label{intdir} \int_{\Pi^+} \vert F' (x+iy)\vert^2 y^{\alpha} \, dx \, dy < \infty ,
\end{equation}
together with the condition that for all $x\in \mathbb{R}$
$$ F(x+iy) \longrightarrow 0, \textnormal{ when } y \longrightarrow \infty .$$
The spaces $ \mathcal{D}_{\alpha}^2  (\Pi^+) $ for $\alpha >-1$ defined in this way actually consists of the unique members of each equivalence class when identifying the functions satisfying \eqref{intdir} that differ by a constant. This being the case, we put $ \Vert F\Vert_{\mathcal{D}_{\alpha}^2 (\Pi^+) } = \Big(\int_{\Pi^+} \vert F' (x+iy)\vert^2 y^{\alpha} \, dx \, dy \Big)^{1/2}$.

\vspace{5pt}

Recall that in the classical Dirichlet space setting ($\alpha =0$) the spectra of \textit{all} linear fractional composition operators $C_{\varphi}$ and $C_{\tau}$ coincide whenever $\tau = g \circ \varphi \circ g^{-1}$ for some conformal map $g: \mathbb{D} \longrightarrow \Pi^+$. (For the spectra of linear fractional composition operators acting on the Dirichlet space $\mathcal{D}^2 (\mathbb{D})$, see \cite{CM, Hu, Hi}.) This, and the boundedness of any linear fractional composition operator acting on the Dirichlet space of the upper half-plane $\mathcal{D}_{0}^2  (\Pi^+)$, follows from the fact that the composition operator $C_g : \mathcal{D}_0^2  (\Pi^+) \longrightarrow \mathcal{D}^2 (\mathbb{D})/ \mathbb{C}$ is unitary. 

\vspace{5pt}

There is also a version of the Paley-Wiener theorem for weighted Dirichlet spaces (see \cite[Thm. 3]{DGM}): Let $\alpha >-1$. Denote by $L^2_{\alpha -1} $ the space consisting of the measurable functions on $\mathbb{R}_{+}$ with finite norm
$$ \Vert f\Vert_{ L^2_{\alpha -1} }=\Big(\frac{\Gamma (\alpha +1)}{2^{\alpha}} \int_0^{\infty} \vert f (t)\vert^2 t^{1-\alpha} \, dt \Big)^{1/2}.$$ 
The Paley-Wiener theorem states that $F\in\mathcal{D}_{\alpha}^2$ if and only if $F(w)=\int_0^{\infty} f(t)e^{iwt} \, dt$ for all $w\in \Pi^+$, where $f \in  L^2_{\alpha -1}$. Moreover, $ \Vert F\Vert_{\mathcal{D}_{\alpha}^2 }=  \Vert f\Vert_{L^2_{\alpha -1}} $.

\vspace{5pt}

By noting that $ \alpha -1 =( \alpha -2) +1$, the above version of Paley-Wiener theorem gives a justification for writing $ \mathcal{D}_{\alpha}^2 (\Pi^+)  = \mathcal{A}_{\alpha -2}^2  (\Pi^+) $. This point of view is also supported by the spectral results below. Note that we consider only those linear fractional composition operators which are bounded also on $  \mathcal{A}_{\alpha}^2 ( \Pi^+)$ for $\alpha > -1$. In fact, as far as we know, it is not known whether there are also other linear fractional transformations that induce bounded composition operators on $ \mathcal{D}_{\alpha}^2 (\Pi^+)$ for $ \alpha \in (-1,0) $ or $\alpha \in (0, 1) $ as there are in the unweighted Dirichlet space.

\begin{lemma}\label{lemmadirsim} Let $\tau : \Pi^+ \longrightarrow \Pi^+$ be a hyperbolic or parabolic map of the form $\tau (w)=\mu w +w_0$, where $\mu >0$ and $\textnormal{Im}\, w_0\geq 0$ (and if $\mu =1$, then $w_0\neq 0$). Then, for all $\alpha >-1$, the operators $ C_{\tau}: \mathcal{D}_{\alpha}^2  (\Pi^+) \longrightarrow \mathcal{D}_{\alpha}^2  (\Pi^+) $ and $ \mu C_{\tau}: \mathcal{A}_{\alpha}^2 ( \Pi^+) \longrightarrow\mathcal{A}_{\alpha}^2 ( \Pi^+) $ are similar.
\end{lemma}

\begin{proof}
It is clear from the definition of the weighted Dirichlet spaces that the differentiation operator $D: F \longmapsto F'$ is an isomorphism $\mathcal{D}_{\alpha}^2  (\Pi^+)  \longrightarrow \mathcal{A}_{\alpha}^2 ( \Pi^+) $. Recall that $C_{\tau} $ is bounded on $\mathcal{A}_{\alpha}^2 ( \Pi^+)  $ for all $\alpha >-1$. Since, moreover, $ (F\circ \tau)' = \tau' (F' \circ \tau)$ and $\tau'(w)=\mu$ for all $w\in\Pi^+$, the composition operator $C_{\tau}$ is bounded on $\mathcal{D}_{\alpha}^2  (\Pi^+) $ and similar to the operator
$$ \mu C_{\tau}:  \mathcal{A}_{\alpha}^2 ( \Pi^+)  \longrightarrow  \mathcal{A}_{\alpha}^2 ( \Pi^+) $$
for any $\alpha >-1$.

\end{proof}

\vspace{5pt}

The following theorems are obvious consequences of Lemma \ref{lemmadirsim} and the spectral results in Sections $3$ and $4$.

\begin{thm}\label{dirpara}
Let $\tau$ be a parabolic self-map of $\Pi^+ $, that is, $\tau (w)=w+w_0$, where $\textnormal{Im}\, w_0 \geq 0$ ($w_0 \neq 0$). Then, for all $\alpha >-1$
\vspace{5pt}

\begin{itemize}
\item[i)] $ \sigma \big( C_{\tau} ; \mathcal{D}_{\alpha}^2  (\Pi^+) \big)=\sigma_{ess} \big( C_{\tau} ; \mathcal{D}_{\alpha}^2  (\Pi^+) \big) = \mathbb{T}$, when $w_0\in \mathbb{R}$,

\vspace{5pt}

\item[ii)] $\sigma \big( C_{\tau} ; \mathcal{D}_{\alpha}^2  (\Pi^+) \big) = \{e^{iw_0t} : t\in [0,\infty ) \} \cup \{0\}$, when $w_0\in \Pi^+$.
\end{itemize}
Moreover, the essential spectrum coincides with the spectrum in both cases.
\end{thm}

\begin{proof}
By Lemma \ref{lemmadirsim}, the operators $C_{\tau} : \mathcal{D}_{\alpha}^2  (\Pi^+)  \longrightarrow \mathcal{D}_{\alpha}^2 (\Pi^+)  $ and $ C_{\tau}: \mathcal{A}_{\alpha}^2 ( \Pi^+)  \longrightarrow \mathcal{A}_{\alpha}^2 ( \Pi^+) $ are similar and hence the result follows from Theorem \ref{sppara}.
\end{proof}

\begin{thm}\label{dirhyper}
Let $\tau$ be a hyperbolic self-map of $\Pi^+ $, that is, $\tau (w)=\mu w + w_0$, where $\mu\in (0,1)\cup (1,\infty)$ and $\textnormal{Im}\, w_0 \geq 0$. Then, for all $\alpha >-1$
\vspace{5pt}

\begin{itemize}
\item[i)] $\sigma \big( C_{\tau};  \mathcal{D}_{\alpha}^2 (\Pi^+)  \big) = \big\{ \lambda \in \mathbb{C} : \vert \lambda \vert = \mu^{ -((\alpha -2) +2 )/2}\big\} =\big\{ \lambda \in \mathbb{C} : \vert \lambda \vert = \mu^{ -\alpha /2}\big\}$, when $w_0\in \mathbb{R}$.

\vspace{5pt}

\item[ii)] $\sigma \big( C_{\tau};  \mathcal{D}_{\alpha}^2  (\Pi^+) \big) = \big\{ \lambda \in \mathbb{C} : \vert \lambda \vert\leq \mu^{ -((\alpha -2) +2 )/2}\big\} =\big\{ \lambda \in \mathbb{C} : \vert \lambda \vert \leq \mu^{ -\alpha /2}\big\} $, when $w_0\in \Pi^+$.
\end{itemize}
Moreover, the essential spectrum coincides with the spectrum in both cases.
\end{thm}

\begin{proof}
Since the operators $C_{\tau} : \mathcal{D}_{\alpha}^2  (\Pi^+)  \longrightarrow \mathcal{D}_{\alpha}^2  (\Pi^+) $ and $ \mu C_{\tau}: \mathcal{A}_{\alpha}^2 ( \Pi^+)  \longrightarrow \mathcal{A}_{\alpha}^2 ( \Pi^+) $ are similar by Lemma \ref{lemmadirsim}, we have that 
$$\sigma \big( C_{\tau};  \mathcal{D}_{\alpha}^2  (\Pi^+) \big) = \sigma \big( \mu C_{\tau};  \mathcal{A}_{\alpha}^2 ( \Pi^+) \big) =\big\{\lambda \in \mathbb{C} : \mu^{-1}\lambda \in \sigma \big( C_{\tau};  \mathcal{A}_{\alpha}^2 ( \Pi^+) \big)\big\}. $$
The result follows from Theorem B in Section \ref{intro} (for the proof, see Theorems \ref{sphypauto}, \ref{hyptauyks}, \ref{hyptaukaks} and Corollary \ref{sphypautoinverse} in Section \ref{hypsec}).
\end{proof}

\section*{Acknowledgements}
This article is part of the author's PhD Thesis and she would like to thank her advisors Hans-Olav Tylli and Pekka Nieminen (University of Helsinki) for helpful discussions and comments. The author is also grateful to Professor Eva Gallardo-Guti\'{e}rrez (Universidad Complutense de Madrid) for inspiring conversations.

\bibliographystyle{plain}

\end{document}